\documentclass[12pt]{amsart}
\usepackage{amssymb,amsmath,verbatim,enumerate,ifthen}
\usepackage[mathscr]{eucal}
\usepackage[utf8]{inputenc}
\usepackage[T1]{fontenc}
\usepackage{cite}
\newtheorem{thm}{Theorem}[section]
\newtheorem{cor}[thm]{Corollary}
\newtheorem{prop}[thm]{Proposition}

\def\eq#1{{\rm(\ref{#1})}}
\def\Eq#1#2{\ifthenelse{\equal{#1}{*}}
  {\[\begin{aligned}[]#2\end{aligned}\]}
  {\begin{equation}\begin{aligned}[]\label{#1}#2\end{aligned}\end{equation}}}

\def\A{\mathscr{A}}
\def\D{\mathscr{D}}
\def\E{\mathscr{E}}
\def\G{\mathscr{G}}

\def\P{\mathscr{P}}
\newcommand\R{\mathbb{R}}
\newcommand\N{\mathbb{N}}
\newcommand\Hc{\mathscr{H}}
\newcommand{\QA}[1]{\A_{#1}}

\title[On the best Hardy constant for means]
{On the best Hardy constant for quasi-arithmetic means and homogeneous deviation means}

\author{Zsolt P\'ales}
\address{Institute of Mathematics, University of Debrecen, Pf.\ 12, 4010 Debrecen, Hungary}
\email{pales@science.unideb.hu}

\author{Pawe\l{} Pasteczka}
\address{Institute of Mathematics, Pedagogical University of Cracow,  Podchor\k{a}\.{z}ych str 2, 30-084 Cracow, Poland}
\email{pawel.pasteczka@up.krakow.pl}

\thanks{The research of the first author was supported by the Hungarian Scientific Research Fund (OTKA) Grant
K-111651 and by the EFOP-3.6.1-16-2016-00022 project. This project is co-financed by the European Union
and the European Social Fund.}



\begin{document}

\begin{abstract} 
The aim of this paper is to characterize the so-called Hardy means, i.e., those means 
$M\colon\bigcup_{n=1}^\infty \R_+^n\to\R_+$ that satisfy the inequality
\Eq{*}{
  \sum_{n=1}^\infty M(x_1,\dots,x_n) \le C\sum_{n=1}^\infty x_n
}
for all positive sequences $(x_n)$ with some finite positive constant $C$. The smallest constant $C$ 
satisfying this property is called the Hardy constant of the mean $M$.

In this paper we determine the Hardy constant in the cases when the mean $M$ is either a concave 
quasi-arithmetic or a concave and homogeneous deviation mean.
\end{abstract}

\maketitle

\section{Introduction}

Hardy's, Landau's, Carleman's and Knopp's celebrated inequalities in an equivalent and unified form state that 
\Eq{H}{
  \sum_{n=1}^\infty \P_p(x_1,\dots,x_n) \le C(p) \sum_{n=1}^\infty x_n,
}
for every sequences $(x_n)_{n=1}^\infty$ with positive terms, where $\P_p$ denotes the $p$-th \emph{power mean} 
(extended to the limiting cases $p=\pm\infty$) and
\Eq{*}{
C(p):=
\begin{cases} 
1 & p=-\infty, \\
(1-p)^{-1/p}&p \in (-\infty,0) \cup (0,1), \\ 
e & p=0, \\
+\infty & p\in[1,\infty],
\end{cases} 
}
and this constant is sharp, i.e., it cannot be diminished. First result of this type with nonoptimal constant was 
established by Hardy in the seminal paper \cite{Har20a}. Later it was improved and extended by Landau \cite{Lan21}, 
Knopp \cite{Kno28}, and Carleman \cite{Car32} whose results are summarized in inequality \eq{H}. More about the history 
of the developments related to Hardy type inequalities is sketched in catching surveys by Pe\v{c}ari\'c--Stolarsky 
\cite{PecSto01}, Duncan--McGregor \cite{DucMcG03}, and in a recent book of 
Kufner--Maligranda--Persson \cite{KufMalPer07}.

In a more general setting, for a given mean $M \colon \bigcup_{n=1}^{\infty} I^n \to I$ (where $I$ is a real interval 
with $\inf I=0$), let $\Hc(M)$ denote the smallest nonnegative extended real number, called the \emph{Hardy 
constant of $M$}, such that
\Eq{*}{
\sum_{n=1}^\infty M(x_1,\dots,x_n)\le \Hc(M)\sum_{n=1}^\infty x_n
}
for all sequences $(x_n)_{n=1}^\infty$ belonging to $I$. If $\Hc(M)$ is finite, then we say that $M$ is 
a \emph{Hardy mean}. In this setup, a $p$-th power mean is a Hardy mean if and only if $p\in[-\infty,1)$ and 
$\Hc(\P_p)=C(p)$ for all $p\in[-\infty,+\infty]$.

For investigating the Hardy property of means, we recall several notions that have been partly introduced and used 
in the paper \cite{PalPas16}. Let $I\subseteq\R$ be an interval and let $M \colon\bigcup_{n=1}^{\infty} I^n \to I$ be 
an arbitrary mean.

We say that $M$ is \emph{symmetric}, \emph{(strictly) increasing}, and \emph{Jensen convex (concave)} if, for all 
$n\in\N$, the $n$-variable restriction $M|_{I^n}$ is a symmetric, (strictly) increasing in each of its variables, and 
Jensen convex (concave) on $I^n$, respectively. It is worth mentioning that means are locally bounded functions, 
therefore, the so-called Bernstein--Doetsch theorem implies that Jensen convexity (concavity) is equivalent to 
ordinary convexity (concavity) (cf.\ \cite{BerDoe15}). If $I=\R_+$, we can analogously define the notion of homogeneity 
of $M$. Finally, the mean $M$ is called 
\emph{repetition invariant} if, for all $n,m\in\N$ and $(x_1,\dots,x_n)\in I^n$, the following identity is satisfied
\Eq{*}{
  M(\underbrace{x_1,\dots,x_1}_{m\text{-times}},\dots,\underbrace{x_n,\dots,x_n}_{m\text{-times}})
   =M(x_1,\dots,x_n).
}

Having all these definitions, let us recall the two main theorems of the paper \cite{PalPas16}.
The first result provides a lower estimation of Hardy constant.

\begin{thm}\label{thm:PalPas16}
Let $I \subset \R_+$ be an interval with $\inf I=0$ and $M \colon \bigcup_{n=1}^{\infty} I^n \to I$ be a mean. Then,
for all non-summable sequences $(x_n)_{n=1}^{\infty}$ in $I$,
\Eq{*}{
\Hc(M) \ge \liminf_{n \to \infty} x_n^{-1} \cdot M\left(x_1,x_2,\ldots,x_n\right).
}
In particular,
\Eq{*}{
\Hc(M) \ge \sup_{y\in I} \liminf_{n \to \infty} \frac ny \cdot M\left(\frac y1,\frac y2,\ldots,\frac yn 
\right).
}
\end{thm}

Under stronger assumptions for the mean $M$, the latter lower estimate obtained above becomes equality by the following 
result.

\begin{thm}\label{thm:mainPaPa}
Let $M \colon \bigcup_{n=1}^{\infty} \R_+^n \to\R_+$ be an increasing, symmetric, repetition invariant, and Jensen 
concave mean. Then
\Eq{H1}{
\Hc(M)=\sup_{y>0} \liminf_{n \to \infty} \frac ny \cdot M\left(\frac y1,\frac y2,\ldots,\frac yn 
\right).
}
If, in addition, $M$ is also homogeneous, then
\Eq{*}{
  \Hc(M)= \lim_{n \to \infty} n \cdot M\left(1,\tfrac 12,\ldots,\tfrac 1n \right),
}
in particular this limit exists.
\end{thm}

Upon taking $M$ to be a power mean in the above theorem, the Hardy--Landau--Knopp--Carleman inequality \eq{H} can 
easily be deduced. For the details, see \cite{PalPas16}.

The purpose of our paper is to find explicit formulas for the constant on the right hand side of equation \eq{H1} in 
two important classes of means: quasi-arithmetic means and homogeneous deviation means. On the other hand we will 
present some iff conditions for these means to satisfy assumptions of theorem above.

\section{Families of means}

From now on our consideration will go twofold. First, we will introduce and recall the most important results 
concerning quasi-arithmetic means. Later, in Section~\ref{sec:fam_dev}, we will do the same with the family of 
deviation means. This splitting will appear in the next section too. 

\subsection{\label{sec:fam_qa}Quasi-arithmetic means}

Idea of quasi-arithmetic means first only glimpsed in a pioneering paper by Knopp \cite{Kno28}. Their theory was 
somewhat later axiomatized in a series of three independent but nearly simultaneous papers by 
De Finetti \cite{Def31}, Kolmogorov \cite{Kol30}, and Nagumo \cite{Nag30} at the beginning of 1930s. 

Let $I$ be an interval and $f \colon I \to \R$ be a continuous, strictly monotone function. For $n \in \N$ and for a 
given vector $x=(x_1,\dots,x_n)\in I^n$, set
\Eq{*}{
  \QA{f}(x):=f^{-1} \left( \frac{f(x_1)+\cdots+f(x_n)}n \right).
}
The mean $\QA{f} \colon \bigcup_{n=1}^{\infty} I^n \to I$ defined this way is called the \emph{quasi-arithmetic mean 
generated by the function $f$}. Quasi-arithmetic means are a natural generalization of power means. Indeed, whenever 
$I=\R_+$ and $f=\pi_p$, where $\pi_p(x):=x^p$ if $p\ne 0$ and $\pi_0(x):=\ln x$, then mean $\QA{f}$ coincides 
with $p$-th power mean (this is what was noticed by Knopp \cite{Kno28}). These means share most of the properties of 
power means. In particular, it is easy to verify that they are symmetric, strictly increasing, and repetition invariant. 
In fact, they admit even more properties of power means (cf. \cite{Kol30}, \cite{Acz48a}).

Immediately after the formal definition, Mulholland \cite{Mul32} characterized the Hardy property in the class of 
quasi-arithmetic means. Namely, he proved that
\Eq{Mul}{
\QA{f} \text{ is a Hardy mean } \iff 
\bigg ( \begin{array}{c}\text{ there exist parameters }q<1 \text{ and }C>0 \\[1mm]
\text{ such that } \QA{f}(x) \le C \cdot \P_q(x) \text{ for all } x\end{array}\bigg).
}

Later, in 1948, Mikusi\'nski \cite{Mik48} proved that the comparability problem within this family can be (under natural smoothness 
assumptions) boiled down to pointwise comparability of the mapping $f\mapsto \frac{f''}{f'}$ (negative of this 
operator is called the \emph{Arrow--Pratt index of absolute risk aversion}; cf.\ \cite{Arr65,Pra64} ). More 
precisely, he proved

\begin{prop}\label{Mik}
Let $I \subset \R$ be an interval, $f,\,g \colon I \to \R$ be twice differentiable functions having nowhere vanishing  
first derivative. Then the following two conditions are equivalent
 \begin{itemize}
 \item[\upshape{(i)}] $\QA{f}(x_1,\dots,x_n)\leq\QA{g}(x_1,\dots,x_n)$ 
 for all $n\in\N$ and vector $(x_1,\dots,x_n)\in I^n$;
 \item[\upshape{(ii)}] $\frac{f''(x)}{f'(x)}\leq \frac{g''(x)}{g'(x)}$ for all $x\in I$.
 \end{itemize}
\end{prop}

Obviously, if $I\subseteq\R_+$, then condition (ii) can be equivalently written as
\Eq{*}{
  \kappa_f(x):=\frac{xf''(x)}{f'(x)}+1\leq\frac{xg''(x)}{g'(x)}+1=:\kappa_g(x) \qquad(x\in I).
}
It is easy to verify (with the notation $\pi_p$ introduced above), that the equality $\kappa_{\pi_p}\equiv p$ 
holds for all $p\in\R$. Therefore, in view of Proposition~\ref{Mik}, we have
\Eq{*}{
\P_q=\QA{\pi_q} \le \QA{f} \le \QA{\pi_p}=\P_p,
}
where $q:=\inf_I\kappa_f$ and $p:=\sup_I\kappa_f$, moreover 
these parameters are sharp. In other words, the operator $\kappa_{(\cdot)}$ could be applied to embed 
quasi-arithmetic means into the scale of power means (cf. \cite{Pas13}). As a trivial consequence we obtain a 
natural estimations of Hardy constants
\Eq{Pow_es}{
C(q) \le \Hc(\QA{f})\le C(p).
}

In the next result we characterize Jensen concave quasi-arithmetic means. Some results in this direction have recently 
been obtained in the paper \cite{CGJJ17}.

\begin{thm}\label{concQA}
Let $f\colon I \to\R$ be a twice continuously differentiable function with a nonvanishing first derivative. 
Then the quasi-arithmetic mean $\QA{f}$ is Jensen concave if and only if either $f''$ is identically zero or $f''$ is 
nowhere zero and the ratio function $\frac{f'}{f''}$ is a convex and negative function on $I$.
\end{thm}

\begin{proof} Without loss of generality, we may assume that $f$ is strictly increasing, then $f'>0$ holds on 
$I$. The Jensen concavity of the mean $\QA{f}$ is equivalent to its concavity, that is, for all $n\in\N$ and for 
all $x=(x_1,\dots,x_n),y=(y_1,\dots,y_n)\in I^n$, $t\in[0,1]$, we have
\Eq{*}{
  \QA{f}(tx+(1-t)y)\geq t\QA{f}(x)+(1-t)\QA{f}(y),
}
in more detailed form,
\Eq{*}{
  &f^{-1}\bigg(\frac{f(tx_1+(1-t)y_1)+\cdots+f(tx_n+(1-t)y_n)}{n}\bigg) \\
  &\qquad\qquad\geq tf^{-1}\bigg(\frac{f(x_1)+\cdots+f(x_n)}{n}\bigg)
       +(1-t)f^{-1}\bigg(\frac{f(y_1)+\cdots+f(y_n)}{n}\bigg).
}
Applying $f$ to this inequality side by side and introducing the new variables $u_i:=f(x_i)$ and 
$v_i:=f(y_i)$, we get that the above inequality is equivalent to
\Eq{*}{
  &\frac{f(tf^{-1}(u_1)+(1-t)f^{-1}(v_1))+\cdots+f(tf^{-1}(u_n)+(1-t)f^{-1}(v_n))}{n} \\
  &\qquad\qquad\geq f\bigg(tf^{-1}\bigg(\frac{u_1+\cdots+u_n}{n}\bigg)
       +(1-t)f^{-1}\bigg(\frac{v_1+\cdots+v_n}{n}\bigg)\bigg)
}
for all $u_1,\dots,u_n,v_1,\dots,v_n\in f(I)=:J$. The meaning of this inequality is exactly the Jensen 
convexity of the function 
\Eq{*}{
  F_t(u,v):=f(tf^{-1}(u)+(1-t)f^{-1}(v))
}
on $J^2$. By the continuity of $F_t$, this property is equivalent to the convexity of $F_t$. In view of the 
regularity 
assumptions of the theorem, this property is satisfied if and only if the second derivative matrix 
$F''_t(u,v)$ is 
positive semidefinite for all $u,v\in J$ and $t\in(0,1)$. This is equivalent to the positive semidefiniteness 
of 
$F''_t(f(x),f(y))$ for all $x,y\in I$ and $t\in(0,1)$. By the Sylvester determinant test, this $2\times2$ 
matrix is 
positive semidefinite if and only if 
\Eq{1T}{
  \partial_1^2 F_t(f(x),f(y))\geq0, \qquad
  \partial_2^2 F_t(f(x),f(y))\geq0
}
and 
\Eq{2T}{
  \partial_1^2 F_t(f(x),f(y))\partial_2^2 F_t(f(x),f(y))-\partial_1\partial_2 F_t(f(x),f(y))^2\geq0.
}
We can easily obtain that 
\Eq{*}{
  \partial_1^2 F_t(f(x),f(y))&=\frac{t^2f''(tx+(1-t)y)}{f'(x)^2}-\frac{tf'(tx+(1-t)y)f''(x)}{f'(x)^3}, \\
  \partial_2^2 F_t(f(x),f(y))&=\frac{(1-t)^2f''(tx+(1-t)y)}{f'(y)^2}
  -\frac{(1-t)f'(tx+(1-t)y)f''(y)}{f'(y)^3}, \\
  \partial_1\partial_2 F_t(f(x),f(y))&=\frac{t(1-t)f''(tx+(1-t)y)}{f'(x)f'(y)}.
}
Therefore, the first inequality in \eq{1T}, is equivalent to 
\Eq{*}{
  \frac{tf''(tx+(1-t)y)}{f'(tx+(1-t)y)}\geq\frac{f''(x)}{f'(x)}.
}
Putting $x=y$, it follows that $f''(x)\leq0$ for all $x\in I$. If, for some $x\in I$, $f''(x)$ were zero, then 
this inequality and the nonpositivity of $f''$ implies that $f''(tx+(1-t)y)=0$ for all $y\in I$ and 
$t\in(0,1)$. Letting $t$ tend to zero, we get that $f''(y)=0$ for all $y\in I$. Therefore, from now on, we may 
assume that $f''$ is strictly negative on $I$.

The inequality in \eq{2T} can now be rewritten in the following form:
\Eq{conv}{
  t\frac{f'(x)}{f''(x)}+(1-t)\frac{f'(y)}{f''(y)}
  \geq\frac{f'(tx+(1-t)y)}{f''(tx+(1-t)y)} \qquad(x,y\in I,\,t\in(0,1)).
}
Thus, we have proved that the ratio function $f'/f''$ is convex and negative.

For the reversed implication, assume that either $f''=0$, or $f''$ is nowhere zero and $f'/f''$ is convex and 
negative. If $f''=0$, then $f(x)=ax+b$ for some real constants $a,b$, hence $\QA{f}$ equals the arithmetic 
mean, which is trivially Jensen concave (and also Jensen convex). 

Without loss of generality, we again may assume  that $f'$ is positive. Suppose now that $f''$ is negative 
and 
$f'/f''$ is convex. As we have seen above, the negativity of the function $f''$ and the convexity of $f'/f''$, 
which 
is expressed by inequality \eq{conv} imply that \eq{2T} is satisfied. Using that $f''$ is negative (recall 
that $f'$ is positive), inequality \eq{conv} yields that
\Eq{*}{
  t\frac{f'(x)}{f''(x)}\geq\frac{f'(tx+(1-t)y)}{f''(tx+(1-t)y)} 
  \quad\mbox{and}\quad
  (1-t)\frac{f'(y)}{f''(y)}\geq\frac{f'(tx+(1-t)y)}{f''(tx+(1-t)y)} 
  \qquad(x,y\in I,\,t\in(0,1)).
}
A simple computation now shows that the two inequalities in \eq{1T} are also fulfilled. Therefore, the second 
derivative matrix of $F_t$ is positive semidefinite on $J^2$, which implies the convexity of $F_t$ for all 
$t\in(0,1)$. However, this property is equivalent to the concavity of the quasi-arithmetic mean $\QA{f}$.
\end{proof}

\subsection{\label{sec:fam_dev}Deviation means}

Given a function $E \colon I \times I \to \R$ vanishing on the diagonal of $I\times I$, continuous and 
strictly decreasing with respect to the second variable, we can define a mean $\D_E\colon 
\bigcup_{n=1}^{\infty} I^n \to I$ in 
the following manner (cf.\ Dar\'oczy \cite{Dar71b}). For every $n\in\N$ and for every vector $x 
=(x_1,\dots,x_n)\in 
I^n$, the \emph{deviation mean (or Dar\'oczy mean)} $\D_E(x)$ is the unique solution $y$ of the equation
\Eq{*}{
E(x_1,y)+\dots+E(x_n,y)=0.
}
By \cite{Pal82a} deviation means are symmetric and repetition invariant. The increasingness of a 
deviation mean $\D_E$ is equivalent to the increasingness of the deviation $E$ in its first variable. All 
these properties and characterizations are consequences of the general results obtained in a series of papers 
by Losonczi \cite{Los70a,Los71a,Los71b,Los71c,Los73a,Los77} (for Bajraktarevi\'c means and Gini means) and by 
Dar\'oczy \cite{Dar71b,Dar72b}, Dar\'oczy--Losonczi \cite{DarLos70}, Dar\'oczy--P\'ales 
\cite{DarPal82,DarPal83} (for deviation means) and by P\'ales 
\cite{Pal82a,Pal83b,Pal84a,Pal85a,Pal88a,Pal88d,Pal88e} (for deviation and quasi-deviation means).

Observe that if $E(x,y)=f(x)-f(y)$ for some continuous, strictly monotone function $f \colon I \to \R$, then 
the deviation mean $\D_E$ reduces to the quasi-arithmetic mean $\QA{f}$. Therefore, deviation means include 
quasi-arithmetic means. One can also notice that Bajraktarevi\'c means and Gini means are also form subclasses 
of deviation means.

It is known \cite{Pal88e} that a deviation mean generated by a continuous deviation function $E\colon 
\R_+^2\to\R$ is homogeneous if and only if $E$ is of the form $E(x,y)=g(y)f(\tfrac xy)$ for some continuous 
functions $f,g\colon\R_+\to\R$ such that $f$ vanishes at $1$ and $g$ is positive. Clearly, the deviation mean 
generated by $E$ is determined only by the function $f$, therefore, as we are going to deal with homogeneous 
deviation means only, let $\E_f$ denote the corresponding deviation mean. In the next section, we will 
determine the Hardy constant for the homogeneous deviation mean $\E_f$ under general circumstances for $f$. 
The following result will be instrumental for our considerations.

\begin{thm}\label{CDM}
Let $f\colon\R_+\to\R$ be a strictly increasing concave function with $f(1)=0$. Then the function 
$E\colon\R_+^2\to\R$ defined by $E(x,y):=f\big(\frac xy\big)$ is a deviation and the corresponding deviation 
mean $\E_f:=\D_E$ is homogeneous, continuous, increasing and concave.
\end{thm}

\begin{proof} It is easy to check, using the continuity (which is a consequence of concavity) and strict 
increasingness 
of $f$ that the function $E$ is a deviation function. The homogeneity of the mean $\E_f$ is obvious. We have 
that $E$ 
is strictly increasing in its first variable, hence, one can show, that the mean $\E_f$ is also strictly 
increasing. 
The continuity of $f$ implies the continuity of the mean directly. We only prove its concavity. 

Let $x=(x_1,\dots,x_n)\in\R_+^n$ and $u=(u_1,\dots,u_n)\in\R_+^n$ and denote $\E_f(x)$ and $\E_f(u)$ by $y$ 
and $v$, 
respectively. Then we have that
\Eq{*}{
  f\Big(\frac{x_1}{y}\Big)+\cdots+f\Big(\frac{x_n}{y}\Big)=0 \qquad\mbox{and}\qquad
  f\Big(\frac{u_1}{v}\Big)+\cdots+f\Big(\frac{u_n}{v}\Big)=0.
}
These equations and the convexity of $f$ follow that
\Eq{*}{
  0&=\frac{y}{y+v}\bigg(f\Big(\frac{x_1}{y}\Big)+\cdots+f\Big(\frac{x_n}{y}\Big)\bigg)
   +\frac{v}{y+v}\bigg(f\Big(\frac{u_1}{v}\Big)+\cdots+f\Big(\frac{u_n}{v}\Big)\bigg)\\
   &=\bigg(\frac{y}{y+v}f\Big(\frac{x_1}{y}\Big)+\frac{v}{y+v}f\Big(\frac{u_1}{v}\Big)\bigg)
     +\cdots+\bigg(\frac{y}{y+v}f\Big(\frac{x_n}{y}\Big)+\frac{v}{y+v}f\Big(\frac{u_n}{v}\Big)\bigg)\\
   &\le f\Big(\frac{y}{y+v}\frac{x_1}{y}+\frac{v}{y+v}\frac{u_1}{v}\Big)
     +\cdots+f\Big(\frac{y}{y+v}\frac{x_n}{y}+\frac{v}{y+v}\frac{u_n}{v}\Big)\\
   &= f\Big(\frac{x_1+u_1}{y+v}\Big)+\cdots+f\Big(\frac{x_n+u_n}{y+v}\Big).
}
Therefore, 
\Eq{*}{
  \frac{\E_f(x)+\E_f(u)}2=\frac{y+v}2\leq \E_f\Big(\frac{x+u}2\Big),
}
which proves that $\E_f$ is Jensen concave. Being also continuous, this property implies the concavity of 
$\E_f$.
\end{proof}

In fact one could strengthen the result of the above theorem, by showing that the concavity of $f$ is not 
only sufficient but also necessary in order that the mean $\E_f$ be concave.

\section{Main Results}

We will apply results contained in \cite{PalPas16} to both quasi-arithmetic and homogeneous deviation means. 
Let us notice that the intersection of these families are power means, which are usually treated as a trivial 
case in a consideration of Hardy property. (See inequality \eq{H}.) 

\subsection{Quasi-arithmetic means}

Let us begin this section with our main result, which allows us to compute the Hardy constant for a large 
class of quasi-arithmetic means.

\begin{thm}
\label{thm:MainQA}
Let $I$ be a real interval with $\inf I=0$ and let $f \colon I \to \R$ be a twice continuously differentiable 
function with nowhere vanishing first derivative. Define
 \Eq{1}
 {
 q:=\liminf_{x \to 0^+} \kappa_f(x) \le \limsup_{x \to 0^+} \kappa_f(x) =:p.
 }
Then, for all $x\in I$,
\Eq{2}{
C(q) \le \liminf_{n \to \infty} \frac nx \QA{f}(\tfrac x1, \tfrac x2 ,\dots, \tfrac xn)
\le \limsup_{n \to \infty} \frac nx \QA{f}(\tfrac x1, \tfrac x2 ,\dots, \tfrac xn) \le C(p)
}
and, consequently, 
\Eq{2+}{
   C(q)\leq\Hc(\QA{f}).
}
\end{thm}

Observe that inequality \eq{2} is the strengthening of \eq{Pow_es}.

\begin{proof}
We prove the right hand side inequality of \eq{2} only, the proof of the left hand side inequality is 
completely analogous, therefore, its detailed proof is left to the reader.

If $p\in[1,+\infty]$, then $C(p)=+\infty$, therefore there is nothing to prove. Assume that $p\in[-\infty,1)$.
Consider any real number $r\in(p,1)\setminus\{0\}$. Then there exists $\delta>0$ such that
\Eq{3}{
\kappa_f(x) \le r \qquad \text{for all}\quad x \in (0,\delta).
}
Let $\varphi \colon I \to \R$ be a $C^2$ function with a non-vanishing first derivative such that, for all 
$x\in I$,
\Eq{4}{
\kappa_\varphi(x)=\max \left(r,\kappa_f(x)\right).
}
(This equation for $\varphi$ is a second-order differential equation, which has such solutions, see e.g. 
\cite{Pas13}.)
By \eq{3}, we have:
\Eq{5}{
\kappa_\varphi(x)=r\qquad \text{for all}\quad x \in (0,\delta).
}
Therefore, there exists constants $\alpha,\beta$ such that $\varphi(x)=\alpha x^{r} + \beta$
for all $x \in (0,\delta)$. 

Define $\varphi_0:=\tfrac{1}{\alpha} \varphi -\beta$. Then $\kappa_\varphi=\kappa_{\varphi_0}$, hence 
equality 
\eq{4} holds with $\varphi_0$, too, furthermore $\QA{\varphi}=\QA{\varphi_0}$. 
Therefore, with no loss of generality, we may assume that $\varphi$ and $\varphi_0$ coincide, i.e.
\Eq{phiapprox0}{
\varphi(x)=x^{r},\qquad x \in (0,\delta].
}
As a consequence, for the inverse function of $\varphi$, we have the formula
\Eq{inv}{
  \varphi^{-1}(y)=y^{1/r} \qquad\mbox{for}\quad y\in 
  \begin{cases}
  (0,\delta^r] &\mbox{ if } r>0,\\ [\delta^r,+\infty) &\mbox{ if } r<0.
  \end{cases}
}
Due to \eq{4}, we have $\kappa_f \le \kappa_\varphi$ so, by Mikusi\'nski's theorem, we get $\QA{f} \le 
\QA{\varphi}$.

Fix now $x\in I$. We have
\Eq{*}{
\limsup_{n \to \infty} \frac nx \QA{f}(\tfrac x1, \tfrac x2 ,\dots, \tfrac xn)
\le \limsup_{n \to \infty} \frac nx \QA{\varphi}(\tfrac x1, \tfrac x2 ,\dots, \tfrac xn)
= \limsup_{n \to \infty} \frac nx \varphi^{-1} \bigg(\frac1n\sum_{i=1}^n\varphi\Big(\frac{x}{i}\Big)\bigg).
}
Now our argument splits into two similar cases depending on the sign of $r$.

Consider first the case $r>0$. By the increasingness of $\varphi$, the sequence $\big(\varphi(\frac xn)\big)$ 
converges 
to zero, therefore there exists $n_0 \in \N$ such that, for all $n>n_0$, 
\Eq{*}{
\frac{1}{n}\sum_{i=1}^{n}\varphi\Big(\frac{x}{i}\Big) <\delta^r.
}
It also implies, for all $n>n_0$, that $\tfrac{x}{n}<\delta$, therefore, applying the construction of $n_0$ 
with 
\eq{phiapprox0} and \eq{inv}, we get  
\Eq{*}{
\frac nx \varphi^{-1}\bigg(\frac{1}{n}\sum_{i=1}^{n}\varphi\Big(\frac{x}{i}\Big) \bigg)
&= \frac nx\bigg(\frac{1}{n}\sum_{i=1}^{n}\varphi\Big(\frac{x}{i}\Big) \bigg)^{1/r}
=\bigg(\frac {n^{r-1}}{x^{r}}\sum_{i=1}^{n}\varphi\Big(\frac{x}{i}\Big)\bigg)^{1/r}\\
&=\bigg(\frac {n^{r-1}}{x^{r}}  \sum_{k=1}^{n_0}\left(\varphi\left(\frac xk\right)-\left(\frac 
xk\right)^{r}\right)+\frac {n^{r-1}}{x^{r}}\sum_{k=1}^{n} \left(\frac xk\right)^{r}\bigg)^{1/r}.\\
}
As $r-1<0$ we obtain that the first term tends to $0$ as $n\to\infty$. The second term equals
\Eq{*}{
\frac {n^{r-1}}{x^{r}}\sum_{k=1}^{n} \left(\frac xk\right)^{r}=\frac{1}{n} 
\sum_{k=1}^{n} \left(\frac kn\right)^{-r} \longrightarrow \int_0^1 t^{-r}dt=\frac{1}{1-r}.
}
Thus 
\Eq{*}{
\limsup_{n \to\infty} \frac nx \QA{f}(\tfrac x1, \tfrac x2 ,\dots, \tfrac xn) \le \limsup_{n \to 
\infty} \frac nx \QA{\varphi}(\tfrac x1, \tfrac x2 ,\dots, \tfrac xn)=(1-r)^{-1/r}=C(r).
}
With appropriate changes in the above argument, we can obtain that the same inequality holds also in  
the case $r<0$. Finally, upon passing the limit $r\to p$, we obtain
\Eq{*}{
\limsup_{n \to \infty} \frac nx \QA{f}(\tfrac x1, \tfrac x2 ,\dots, \tfrac xn) \le C(p).
}
Completely analogous considerations lead to the inequality
\Eq{*}{
\liminf_{n \to \infty} \frac nx \QA{f}(\tfrac x1, \tfrac x2 ,\dots, \tfrac xn) \ge C(q).
}
Finally, using this inequality and the second inequality of Theorem~\ref{thm:PalPas16}, it follows that 
\eq{2+} is also valid. Thus the proof is complete.
\end{proof}

Combining Proposition~\ref{Mik} with Theorem~\ref{thm:MainQA} we 
immediately obtain

\begin{thm}
\label{thm:MainQA1}
Let $f \colon \R_+ \to \R$ be a twice continuously differentiable function with nowhere vanishing first 
derivative. If the limit $p:=\lim_{x \to 0^+} \kappa_f(x)$ exists and $\kappa_f(x)\le p$ for every $x>0$, 
then 
$\Hc(\QA{f})=C(p)$. In particular, this mean is Hardy if and only if $p<1$. 
\end{thm}

\begin{proof}
By 
\eqref{Pow_es} we get $\Hc(\QA{f})\le C(\sup \kappa_f(x))=C(p)$.
The converse inequality is implied by Theorem~\ref{thm:MainQA} -- more precisely by inequality \eqref{2+}.
\end{proof}

The following result strengthens Mulholland's theorem \cite{Mul32}  in the class of concave quasi-arithmetic 
means (compare \eq{Mul} and the implication $\text{(i)} \Rightarrow \text{(ii)}$ below).

\begin{thm}
\label{thm:MainQA2}
Let $f \colon \R_+ \to \R$ be a twice continuously differentiable function with nowhere vanishing first and 
second derivatives such that $f'/f''$ is convex and negative. Then $\kappa_f$ is a decreasing function. 
Furthermore, the following assertions are equivalent:
\begin{enumerate}[(i)]
 \item $\QA{f}$ is a Hardy mean;
 \item There exists a parameter $q<1$ such that $\QA{f}\leq\P_q$;
 \item $p:=\lim_{x \to 0^+} \kappa_f(x)<1$.
\end{enumerate}
And, in each of the above cases, $\QA{f}\leq\P_p$, and $\Hc(\QA{f})=C(p)$.
\end{thm}

\begin{proof} 
By the convexity and negativity of the function $f'/f''$, the mapping $x\mapsto f'(x)/(x \cdot f''(x))$ is 
negative and increasing. Therefore, $\kappa_f$ is decreasing, whence the right limit $p$ of $\kappa_f$ at 
zero 
exists and $\kappa_f\leq p$. Then, by Proposition~\ref{Mik}, it follows that $\QA{f} \le \P_p$.

To prove $(i) \Rightarrow (ii)$, assume that $\QA{f}$ is a Hardy mean. 
Then Theorem~\ref{thm:MainQA1} implies that $p<1$, hence $(ii)$ holds with $q=p$.
The implication $(ii) \Rightarrow (iii)$ is obvious. 

Finally, we prove $(iii) \Rightarrow (i)$. In view of Theorem~\ref{concQA}, the convexity and negativity of 
$f'/f''$  implies that $\QA{f}$ is a concave quasi-arithmetic mean. Therefore, by Theorems~\ref{thm:mainPaPa} 
and \ref{thm:MainQA},
\Eq{*}{
\Hc(\QA{f})
=\sup_{y>0} \liminf_{n \to \infty} \frac ny \cdot \QA{f}\left(\frac y1,\frac y2,\ldots,\frac yn\right)
\leq \sup_{y>0} \limsup_{n \to \infty} \frac ny \cdot \QA{f}\left(\frac y1,\frac y2,\ldots,\frac yn\right)
\leq C(p),
}
where $p<1$, hence $\QA{f}$ is a Hardy mean.

To complete the proof, observe that, by the convexity and negativity of $f'/f''$, the mean $\QA{f}$ is 
concave, therefore Theorem~\ref{thm:MainQA1} implies the equality $\Hc(\QA{f})=C(p)$.
\end{proof}

\subsection{Homogeneous deviation means}

\begin{thm}\label{CDM+}
Let $f\colon\R_+\to\R$ be a strictly increasing concave function with $f(1)=0$. Then the homogeneous deviation 
mean $\E_f$ is a Hardy mean if and only if 
\Eq{HC}{
  \int_0^1f\Big(\frac{1}{t}\Big)dt<+\infty
}
and, if the above inequality holds, then its Hardy constant is the unique positive solution $c$ of the equation
\Eq{*}{
  \int_0^cf\Big(\frac{1}{t}\Big)dt=0.
}
\end{thm}

\begin{proof}
By Theorem \ref{CDM}, the deviation mean $\E_f$ is homogeneous, continuous, increasing and concave.
Therefore, applying the second assertion of Theorem \ref{thm:mainPaPa}, we have that
\Eq{*}{
  \Hc(\E_f)=\lim_{n\to\infty} \E_f\Big(\frac{n}1,\frac{n}2,\dots,\frac{n}n\Big).
}
Assume first that $\E_f$ possesses the Hardy property. Then $\Hc(\E_f)$ is finite, and hence the sequence
$\big(\E_f\big(\tfrac{n}1,\tfrac{n}2,\dots,\tfrac{n}n\big)\big)$ is bounded. Thus, there exists $K>0$ such that,
for all $n\in\N$,
\Eq{Efn}{
  \E_f\Big(\frac{n}1,\frac{n}2,\dots,\frac{n}n\Big)\leq K,
}
which is equivalent to 
\Eq{*}{
  f\Big(\frac{n}{K}\Big)+f\Big(\frac{n}{2K}\Big)+\cdots+f\Big(\frac{n}{nK}\Big)\leq0
}
For the sake of brevity define the function $f^*\colon\R_+\to\R$ by $f^*(t):=f\big(\frac{1}{t}\big)$. Then $f^*$ is a 
continuous strictly decreasing function with $f^*(1)=0$. Hence, the above inequality is equivalent to 
\Eq{im}{
  \frac{K}{n}\bigg(f^*\Big(\frac{K}{n}\Big)+f^*\Big(\frac{2K}{n}\Big)+\cdots+f^*\Big(\frac{nK}{n}\Big)\bigg)\leq0.
}
By the decreasingness of $f^*$, for all $i\in\{1,2,\dots,n\}$, we have that
\Eq{*}{
  \int_{\frac{iK}{n}}^{\frac{(i+1)K}{n}}f^*(t)dt\leq \frac{K}{n}f^*\Big(\frac{iK}{n}\Big).
}
After adding up these inequalities side by side and using \eq{im}, we get
\Eq{Kn}{
  \int_{\frac{K}{n}}^{\frac{(n+1)K}{n}}f^*(t)dt\leq0.
}
Hence, for all $n\geq K$,
\Eq{*}{
  \int_{\frac{K}{n}}^1f^*(t)dt\leq -\int_1^{\frac{(n+1)K}{n}}f^*(t)dt.
}
Upon taking the limit $n\to\infty$, we arrive at the inequality
\Eq{*}{
  (0\leq) \int_{0}^1f^*(t)dt\leq -\int_1^Kf^*(t)dt,
}
which proves that condition \eq{HC} must be valid.

Now assume that \eq{HC} holds. Define the function $F\colon\R_+\to\R$ by
\Eq{*}{
  F(x):=\int_0^xf\Big(\frac{1}{t}\Big)dt=\int_0^xf^*(t)dt.
}
Obviously, $F$ is continuous, strictly increasing on $(0,1)$, and strictly decreasing on $(1,\infty)$. To show that it 
has a zero on $(1,\infty)$, it suffices to prove that $F(x)$ tends to $-\infty$ as $x\to+\infty$.

By the concavity of $f$, there exists a positive constant $a\in\R$ such that, for all $x\in\R_+$,
\Eq{*}{
  f(x)=f(x)-f(1)\leq a(x-1).
}
Therefore, for $x\geq 1$, we get
\Eq{*}{
  F(x)=\int_0^xf\Big(\frac{1}{t}\Big)dt
      \leq \int_0^1f\Big(\frac{1}{t}\Big)dt + a\int_1^x\Big(\frac{1}{t}-1\Big)dt
      \leq \int_0^1f\Big(\frac{1}{t}\Big)dt + a(\ln(x)-x+1).
}
The right hand side estimate tends to $-\infty$ as $x\to+\infty$, therefore $F$ also has this property.
This ensures that $F$ has a unique zero, denoted by $c$, in the interval $(1,\infty)$. 

In the rest of the proof, we show that $c=\Hc(\E_f)$. Let $K>\Hc(\E_f)$ be an arbitrary number. Then there exists $n_0$ 
such that, for all $n\geq n_0$, we have the inequality \eq{Efn}. Repeating the same argument as above, this inequality 
implies \eq{Kn} for all $n\geq n_0$. Upon taking the limit $n\to\infty$, we obtain that 
\Eq{*}{
  F(K)=\int_0^K f^*(t)dt\leq0.
}
Consequently, $c\leq K$. Taking now the limit $K\to\Hc(\E_f)$, we get that $c\leq\Hc(\E_f)$.

The proof of the reversed inequality $c\geq\Hc(\E_f)$ is completely analogous, therefore the equality $c=\Hc(\E_f)$ 
holds as desired.
\end{proof}

In order to formulate a corollary of this theorem for Gini means (cf.\ \cite{Gin38}), we introduce the 
following notation: Given two real numbers $p,q\in\R$, define the function $\chi_{p,q}\colon\R_+\to\R$ by
\Eq{*}{
  \chi_{p,q}(x)
  :=\begin{cases}
    \dfrac{x^p-x^q}{p-q} & \mbox{ if } p\neq q, \\[4mm]
    x^p\ln(x) & \mbox{ if } p=q.
    \end{cases}
}
In this case, the function $E_{p,q}\colon \R_+^2\to\R$ defined by 
\Eq{*}{
  E_{p,q}(x,y):=y^p\chi_{p,q}\Big(\frac{x}{y}\Big)
}
is a deviation function on $\R_+$. The deviation mean generated by $E_{p,q}$ will be denoted by 
$\G_{p,q}$ and called the Gini mean of parameter $p,q$ (cf.\ \cite{Gin38}). One can easily see that 
$\G_{p,q}$ has the following explicit form:
\Eq{GM}{
  \G_{p,q}(x_1,\dots,x_n)
   :=\left\{\begin{array}{ll}
    \left(\dfrac{x_1^p+\cdots+x_n^p}
           {x_1^q+\cdots+x_n^q}\right)^{\frac{1}{p-q}} 
      &\mbox{if }p\neq q, \\[4mm]
     \exp\left(\dfrac{x_1^p\ln(x_1)+\cdots+x_n^p\ln(x_n)}
           {x_1^p+\cdots+x_n^p}\right) \quad
      &\mbox{if }p=q.
    \end{array}\right.
}
Clearly, in the particular case $q=0$, the mean $\G_{p,q}$ reduces to the $p$th H\"older mean $\P_p$. It is 
also obvious that $\G_{p,q}=\G_{q,p}$.

\begin{cor} 
Gini means $\G_{p,q}$ is increasing and concave if and only if 
\Eq{pq}{
\min(p,q)\leq0\leq\max(p,q)\leq 1.
}
In this case, $\G_{p,q}$ is a Hardy mean if and only if $\max(p,q)<1$ and then
\Eq{*}{
\Hc(\G_{p,q})=\begin{cases} 
               \left(\dfrac{1-p}{1-q}\right)^{\frac1{q-p}} & p \ne q, \\
               e & p=q=0.
              \end{cases}
}
\end{cor}

\begin{proof}
By the results of Losonczi \cite{Los71a}, \cite{Los71b}, $\G_{p,q}$ is concave if and only if \eq{pq} holds.
On the other hand, $\G_{p,q}$ is increasing if and only if the first two inequalities in \eq{pq} are 
satisfied, that is, if $0$ is between $p$ and $q$.

Now, for $p \ne q$, in view of Theorem~\ref{CDM+}, with $f=\chi_{p,q}$, we get that $\G_{p,q}$ is a Hardy 
mean if and only if 
\Eq{*}{
  \int_0^1 \chi_{p,q}\Big(\frac1t\Big) dt<+\infty.
}
For $p\neq q$, we have
\Eq{*}{
\int_0^1 \chi_{p,q}\Big(\frac1t\Big)\, dt
=\frac1{p-q} \int_0^1 (t^{-p}-t^{-q})\, dt, 
}
which is finite if and only if $\max(p,q)<1$. In that case, the Hardy constant $c$ of the mean $\G_{p,q}$
satisfies
\Eq{*}{
  0=\int_0^c \chi_{p,q}\Big(\frac1t\Big)\, dt
   =\frac1{p-q} \int_0^c (t^{-p}-t^{-q})\, dt
   =\frac{1}{p-q} \left( \frac{1}{1-p} c^{1-p} - \frac{1}{1-q} c^{1-q} \right).
}
Solving this equation with respect to $c$, we obtain that
\Eq{*}{
  c = \left( \frac{1-p}{1-q} \right)^{\frac1{q-p}}.
}
In the case $p=q=0$, we get 
\Eq{*}{
  \int_0^1 \chi_{0,0}\Big(\frac1t\Big)\, dt=-\int_0^1 \ln t \, dt=1<+\infty,
}
proving that $\G_{0,0}$ is a Hardy mean. For its Hardy constant $c$ it remains to find the positive solution 
of the equation 
\Eq{*}{
  0=\int_0^c \chi_{0,0}\Big(\frac1t\Big)\, dt=-\int_0^c \ln t \, dt=c(1-\ln c),
}
which results the solution $c=e$.
\end{proof}


\end{document}